\documentclass[12pt]{amsart}
\usepackage{amsmath,amsthm,amsfonts,amssymb}
\begin{document} 
\newcommand{\A}{{\mathbb A}}
\newcommand{\Ad}{{\bold A}}
\newcommand{\Af}{{\bold A}_{f}}
\newcommand{\B}{{\mathbb B}}
\newcommand{\C}{{\mathbb C}}
\newcommand{\N}{{\mathbb N}}
\renewcommand{\O}{{\mathcal O}}
\newcommand{\Q}{{\mathbb Q}}
\newcommand{\Z}{{\mathbb Z}}
\renewcommand{\P}{{\mathbb P}}
\newcommand{\R}{{\mathbb R}}
\newcommand{\rc}{\subset}
\newcommand{\rank}{\mathop{rank}}
\newcommand{\trace}{\mathop{tr}}
\newcommand{\Spec}{\mathop{Spec}}
\newcommand{\dimc}{\mathop{dim}_{\C}}
\newcommand{\Lie}{\mathop{Lie}}
\newcommand{\Auto}{\mathop{{\rm Aut}_{\mathcal O}}}
\newcommand{\alg}[1]{{\mathbf #1}}
\newtheorem{theorem}{Theorem}[section]
\newtheorem{definition}[theorem]{Definition}
\newtheorem{corollary}[theorem]{Corollary}
\newtheorem{conjecture}[theorem]{Conjecture}
\newtheorem{example}[theorem]{Example}
\newtheorem{remark}[theorem]{Remark}
\newtheorem{remarks}[theorem]{Remarks}
\newtheorem{lemma}[theorem]{Lemma}
\newtheorem{proposition}[theorem]{Proposition}
\title[Entire Curves and Integral Sets]{%
Entire Curves, Integral Sets    and 
Principal Bundles.
}
\author {J\"org Winkelmann}
\begin{abstract}
We compare the behaviour of entire curves and integral sets, in
particular in relation to locally trivial fiber bundles,
algebraic groups and finite ramified covers over semi-abelian varieties.
\end{abstract}
% Keywords: entire curves, integral point sets,
% semi-abelian variety, algebraic group
\address{%
J\"org Winkelmann \\
Mathematisches Institut\\
Universit\"at Bayreuth\\
Universit\"atsstra\ss e 30\\
D-95447 Bayreuth\\
Germany\\
}
\email{jwinkel@member.ams.org\newline\indent{\itshape Webpage: }%
http://btm8x5.mat.uni-bayreuth.de/\~{ }winkelmann/
}
\maketitle
\section{Survey}
It has been widely conjectured by Serge Lang and other mathematicians
that the diophantine behaviour of a projective
variety $V$ defined over a number
field $K$ is related to its complex-analytic properties.
The philosophy is that
for a projective variety defined over a number field $K$ infinite sets
of points rational over some finite extension of $K$ should
correspond to holomorphic entire curves.

For non-compact varieties there is an analoguos philosophy:
Entire curves should correspond to infinite sets of integral points.
(Loosily speaking integral points are points whose coordinates
are integers. See discussion below.)

\begin{conjecture}
Let $X$ be an algebraic variety defined over a number field $K$.

Then there exists a non-constant holomorphic map $f:\C\to X(\C)$
(with Zariski dense image) if and only if there is a finite
field extension $K'/K$ such that $X$ admits an infinite
(resp.~Zariski dense) integral set.
\end{conjecture}

This is known to be true in dimension one: Using the uniformization
theorem or Nevanlinna theory one can prove that a complex curve
(``Riemann surface'') $X$ admits a non-constant holomorphic map
$f:\C\to X$ if and only if $X$ is biholomorphic to $\P_1$, $\C$,
$\C^*$ or an elliptic curve, while on the other side the theorem
of Siegel (for affine curves) and Faltings proof of
``Mordells conjecture'' (for projective curves) imply that these
are also the only algebraic curves for which the above stated
arithmetic
analogue holds.

It is also known to be true for subvarieties of abelian varieties.
In fact, for a subvariety $Z$ of an abelian variety $A$ both
properties are equivalent to the assumption that there is a
translate $W$ of an abelian subvariety of $A$ which is contained
in $Z$.

However, for arbitrary varieties in dimension two or higher this
conjecture is wide open.

The purpose of this article is to provide some small steps
towards this conjecture.
We will discuss fiber bundles, algebraic groups and ramified
coverings over semi-abelian varieties from this point of view.

We verify that integral points sets and entire curves
do share some common functorial properties concerning these
topics.

In particular we establish some lifting properties for fiber bundles
which we use in \S1.4 in a combination with a variant of
Jouanolou's trick to provide a new way to define the notion of an
``integral set''.

\subsection{Fiber bundles}
If $f:\C\to X$ is an entire curve and $E\to X$ a holomorphic fiber
bundle, then we can lift $f$ to a map $F:\C\to E$. We prove a strong 
complex-analytic
statement
in this direction and
an arithmetic analog.

\begin{theorem}
Let $G$ be an algebraic group, $X$ an algebraic variety
and $\pi:E\to X$ a $G$-principal bundle,
all defined over a number field $K$.

Then the following assertions hold true:

\begin{itemize}
\item[a)]
For every holomorphic map $f:\C\to X(\C)$ there exists a
holomorphic $F:\C\to E(\C)$ such that
\begin{enumerate}
\item $f=\pi\circ F$,
\item Assume furthermore that $G$ is connected.
Then $F$ can be chosen such that the
 Zariski closure of the image $F(\C)$ coincides with
the preimage $\pi^{-1}(\overline{f(\C)}^{Zar})$ of the Zariski
closure of the image of $f$.
\end{enumerate}

\item[b)]
Assume furthermore that the bundle $E\to X$ is locally trivial in the Zariski
topology.
Let $R$ be a set of integral points in $X(K)$.
Then there exists a finite field extension $K'/K$ and a set of
integral points $R'\subset E(K')$ such that
\begin{enumerate}
\item
$\pi(R')=R$.
\item
The Zariski closure of $R'$ in $E$ coincides with 
the preimage $\pi^{-1}(\overline{R}^{Zar})$ of the Zariski
closure of $R$.
\end{enumerate}
\end{itemize}
\end{theorem}

{\sl Remarks.} 
\begin{enumerate}
\item
The arithmetic statement should be true even without
the additional assumption of $E\to X$ being Zariski locally trivial.
\item No extension $K'/K$ of the number field $K$ is needed if we keep the
assumption of $E\to X$ being Zariski locally trivial and assume
in addition that $G$ admits a Zariski dense integral subset.
\item
We can not simulataneously drop the condition of being locally trivial
in the Zariski topology {\em and} do without extending the field.
In fact, for an arbitrary $G$-principal bundle $E\to X$ it may
happen that $E(K)$ is empty even if $X(K)$ contains an integral
subset dense in the Zariski topology: The finite group $\Z/2\Z$
is a real algebraic group as $\{z\in\R^*:z^2=1\}=\Spec\R[t]/(t^2-1)$ 
acting on its
principal homogeneous space $E=\{z\in\C^*:z^2=-1\}$.
Then $E(\R)$ is empty and $E\to\Spec\R$ is a $G$-principal bundle.
\end{enumerate}

\subsection{Groups}
For a complex Lie group $G$ it is easy to construct holomorphic maps
$f:\C\to G$ due to the existence of the exponential map.

\begin{theorem}
Let $G$ be a connected complex Lie group.
Then there exists a holomorphic map $f:\C\to G$ with dense image.
\end{theorem}

\begin{theorem}
Let $G$ be an algebraic group defined over a number field $K$.

Then there exists a finite field extension $K'/K$ and a integral
set $R\subset G(K')$ which is dense in the Zariski topology.
\end{theorem}

({\em Remark.} This is of course well-known, 
but we may also regard this as a special case theorem~1, if we consider $G$
as
trivial $G$-principal bundle over a point.)

\subsection{Ramified coverings over semi-abelian varieties}
A semi-abe\-lian variety $A$ is a group variety which admits
a short exact sequence of algebraic groups
\[
1 \to T \to A \to M \to 1 
\]
where $T$ is a torus and $M$ is an abelian variety.

We regard varieties admitting a finite morphism onto a
semi-abelian variety.

\begin{theorem}\label{thm-ram}
Let $X$ be a variety defined over a number field $K$ which admits
a finite morphism $\pi$ onto a semi-abelian variety $A$.

Assume that there exists a 
holomorphic map $f:\C\to X$ which is non-constant resp.~with Zariski
dense image.

Then there exists a finite field extension $K'/K$ and an integral
set $R\subset X(K')$ such that $R$ is infinite
resp.~Zariski dense.
\end{theorem}

\begin{corollary}
Let $X$ be a variety defined over a number field $K$ which admits
a finite morphism onto a semi-abelian variety.

Assume that every integral subset of $X(K')$
is finite for a every finite field extension $K'/K$.

Then every holomorphic map from $\C$ to $X(\C)$ is constant.
\end{corollary}

\subsection{Characterising integral sets}
For an affine variety $V$ there is an easy definition of
an ``integral'' set:
This is a set $H$ of $K$-rational points on $V$ such that there
exists a closed embedding $i$ of $V$ into an affine space $\A^n$
such that all the coordinates of all the elements $i(x)$ ($x\in H$)
are integral (i.e.~in some $\O_L(S)$ where $L$ is a finite extension
of $K$ and $S$ a finite set of valuations including all archimedean
ones).

For arbitrary varieties there are two equivalent (see \cite{V}, proposition~1.4.7)
definitions: The first definition uses models over the integers,
the second Weil functions.

Here we present a third way, which we feel is more elementary.
\begin{proposition}
Let $V$ be a quasi-projective variety defined over some number field $K$.
Then $H\subset V(K)$ is integral if and only if
there exists a finite field
extension $L/K$ and a finite set $S$ of valuations of $L$
including all archimedean ones, an affine $L$-variety $W$, a
$L$-morphism $\phi:W\to V$ and an integral subset $I\subset W(L)$
such that $H\subset\phi(I)$.
\end{proposition}

\begin{proof}
One direction is well-known (see \cite{V}).
The other direction follows from our Main Theorem in combination with the
two lemmata below which are essentially a variant of a result
known as ``Jouanolou's trick'', see \cite{J}.
\end{proof}

\begin{lemma}
Let $V$ be a quasiprojective variety. Then there is a
$G_m$-principal bundle $Q\to V$ such that $Q$ is a quasi-affine
variety.
\end{lemma}

\begin{proof}
We embed $V\subset \P_N$ and take $Q=\pi^{-1}(V)$ where
$\pi:\A^{N+1}\setminus\{(0,\ldots,0)\}\to\P_N$.
\end{proof}

\begin{lemma}
Let $Q$ be quasi-affine variety.
Then there is a principal bundle $W\to Q$ with a linear algebraic
group as structure group such that $W$ is an affine variety.
\end{lemma}

\begin{proof}
We realize $Q$ as $Q=\bar Q\setminus Z$ where $\bar Q$ is an affine
variety and $Z$ is a closed subvariety.
Then there are regular functions $f_1,\ldots,f_s$ on $\bar Q$ such that
$Z$ is the joint zero locus of the $f_i$. Then $z\mapsto (z,f_1(z),
\ldots, f_s(z))$ yields a closed embedding $i$ of $Q$ into 
\[
\bar Q\times \left( \A^s\setminus\{(0,\ldots,0)\} \right).
\]
The usual action of $SL_s$ on $\A^s$ gives us a principal
bundle $SL_s\to  \left( \A^s\setminus\{(0,\ldots,0)\} \right)$
and thereby a principal bundle $\rho$ from $\bar Q\times SL_s$
to
\[
\bar Q\times \left( \A^s\setminus\{(0,\ldots,0)\} \right).
\]
Now we take the restriction of $\rho$ to $W=\rho^{-1}(i(Q))$.
Note that $W$ is an affine variety, because it is a closed
subvariety of $\bar Q\times SL_s$ and both $\bar Q$ and $SL_s$ are
affine varieties.
\end{proof}

\section{Complex Lie groups}

In this section we will prove the analytic part of theorem~1.

Our first step is the auxiliary result below.

\begin{proposition}\label{cn-Lie-surj}
Let $G$ be a connected complex Lie group.
Then there exists a complex vector space $\C^N$
with a surjective holomorphic map $f:\C^N\to G$.
\end{proposition}

\begin{proof}
Let $\pi:\tilde G\to G$ be the universal covering and
$\tilde G=S\cdot R$ its Levi-Malcev decomposition, i.e., $S$ is a maximal
connected semisimple Lie subgroup and $R$ is a maximal connected normal
solvable subgroup. As a semisimple complex Lie group, $S$ is algebraic. 
Fixing a root system, we obtain a Borel subgroup $B$ corresponding
to the positive roots and an ``opposite Borel subgroup'' $B^-$
corresponding to the negative roots.
Now $B^-\cdot B$ contains a Zariski open neighbourhood $V$ of
$e$ in $S$. Since connected topological groups are generated as a
group by
each neighbourhood of the neutral element, we have
$\cup_k V^k=S$. Since $S\setminus V^k$ is a descending sequence
of closed subvarieties, it follows that $(B^-\cdot B)^k=S$
for some $k\in\N$.
Thus there is a surjective holomorphic map from the connected
solvable complex Lie group $A=(B^-\times B)^k\times R$ to $G$.
This implies the statement, because for every connected
solvable complex Lie group $A$ the universal covering space $\tilde A$
is biholomorphic to some $\C^N$.
\end{proof}

\begin{corollary}
Let $G$ be a connected complex Lie group.

Then there exists a holomorphic map $f:\C\to G$
with dense image.
\end{corollary}

\begin{proof}
This follows from the theorem, because there exists a 
holomorphic map $f:\C\to\C^N$ with dense image for every $N\in\N$
(see e.g.~\cite{W1}).
\end{proof}

Next we need an auxiliary result.
\begin{proposition}\label{disc-dense}
Let $C$ be a connected complex manifold, 
$\rho:C\to\R^+$ an unbounded continuous
function, $X$ a complex algebraic variety and $f:C\to X$ a holomorphic
map such that the image $f(C)$ is Zariski dense.

Then there exists a discrete subset $D\subset C$ such that $f(D)$ is Zariski
dense in $X$. Moreover $D$ can be chosen such that $\{x\in D:\rho(x)<c\}$
is finite for every $c\in\R$.
\end{proposition}

\begin{corollary}
Let $C$ be a complex manifold and $F:C\to X$ be a holomorphic map
to a complex algebraic variety. Then there exists a discrete subset
$D\subset C$ such that the Zariski closures of $f(C)$ and $f(D)$
coincide.
\end{corollary}

\begin{proof}
Let ${\mathcal I}$ be the set of all discrete subsets $D\subset C$ for which
\begin{enumerate}
\item
all the sets $\{x\in D:\rho(x)<c\}$ are finite and
\item
the Zariski closure of $f(D)$ in $X$ is irreducible.
\end{enumerate}
Then we define ${\mathcal M}\subset{\mathcal I}$ as the family of all
those $D\in{\mathcal I}$ which are maximal in the following sense:
{\em If $D'\in{\mathcal I}$ with 
$\overline{f(D)}^{Zar}\subset\overline{f(D')}^{Zar}$,
then
$\overline{f(D)}^{Zar}=\overline{f(D')}^{Zar}$.}

We claim that there are only finitely many subvarieties $V\subset X$
arising as the Zariski closure of an image $f(D)$ for some 
$D\in {\mathcal M}$. Indeed, let us assume the converse and let
$D_n$ be an infinite sequence in ${\mathcal M}$ for which
the subvarieties $\overline{f(D_n)}^{Zar}$ are all different.
Then we define $\Delta\subset C$ via
\[
\Delta=\cup_{n\in\N}\{x\in D_n:\rho(x)>n\}
\]
By construction $\{x\in\Delta:\rho(x)<c\}$ is finite for every $c\in\R$.
This implies that $\Delta$ is discrete. For each $n\in\N$ the set $\Delta$
contains all but finitely many elements of $D_n$.
Since $\overline{f(D_n)}^{Zar}$ is irreducible, it follows that
\[
\overline{f(D_n)}^{Zar}=\overline{f(D_n\cap\Delta)}^{Zar}\subset
\overline{f(\Delta)}^{Zar}
\]
for each $n\in\N$.
Now let $Z$ be an irreducible component of $\overline{f(\Delta)}^{Zar}$.
Then $Z\cap\Delta\in{\mathcal I}$. 
By the maximality assumption for the
$D_n$ it follows that $\overline{f(D_n)}^{Zar}$ is an irreducible
component of $\overline{f(\Delta)}^{Zar}$ for each $n\in\N$. Since
$\overline{f(\Delta)}^{Zar}$ has only finitely many irreducible components,
we deduce that there exist only finitely many subvarieties
$V\subset X$ realizable as $V=\overline{f(D)}^{Zar}$ for some $D\in
{\mathcal M}$.

Now let $Z_0$ be the union of all such $V$. Since there are only finitely
many such $V$, the union $Z_0$ is a subvariety.

Next we claim that $Z_0=X$.
Indeed, if not, then $f^{-1}(X\setminus Z_0)$ is a dense open subset of $C$
which allows as to choose an infinite subset $D'\subset C$ such that
$f(D')\cap Z_0=\{\}$ and such that $\{x\in D':\rho(x)<c\}$ is finite
for every $c\in\R$. Now for every irreducible component $W$
of $\overline{f(D')}^{Zar}$ we obtain an element $W\cap D'\in{\mathcal I}$.
By noetherianity there must exist an element $D\in{\mathcal M}$ with
$\overline{f(W\cap D')}^{Zar}\subset\overline{f(D)}^{Zar}$ which leads to a
contradiction, because $f(D')$ does not intersect $Z_0$.
Thus $Z_0=X$. Since $X=\overline{f(C)}^{Zar}$ is irreducible, it follows
that $X=\overline{f(D)}^{Zar}$ for some $D\in{\mathcal M}$.
\end{proof}

\begin{theorem}
Let $G$ be a complex algebraic group and $\pi:E\to X$ 
be a complex algebraic $G$-principal bundle
(locally trivial in the \'etale topology).

Then for every holomorphic map $f:\C\to X(\C)$ there exists a
holomorphic map $F:\C\to E(\C)$ such that
\begin{enumerate}
\item $f=\pi\circ F$,
\item The Zariski closure of the image $F(\C)$ coincides with
the preimage $\pi^{-1}(\overline{f(\C)}^{Zar})$ of the Zariski
closure of the image of $f$.
\end{enumerate}
\end{theorem}

\begin{proof}
There is no loss in generality in assuming that $f(\C)$ is
Zariski dense in $X$.

The pull back bundle $E\times_X\C\to\C$
is a holomorphic $G$-principal bundle over $\C$. It is topologically
trivial, because $\C$ is a contractible topological space.
Due to Grauert's Oka principle (\cite{Gr}) this implies that it
is holomorphically trivial and therefore admits a holomorphic section
$\sigma:\C to E\times_X\C$.

On the other hand, $X$ contains a Zariski open subset $U$ which admits
an \'etale cover $V\to U$ such that the pull-back bundle of $E$ is trivial
over $V$.
Define $U'=f^{-1}(U)$ and $V'=V\times_U U'$. 
Now we can choose two trivializations of the pull-back bundle over $V'$:
First, we choose a holomorphic section 
 $\sigma:\C\to E\times_X\C$ which yields a holomorphic section
 $\sigma'$
of $E'=E\times_XV'\to V'$. Second, the assumption that
$E\times_XV\to V$ is algebraically trivial permits us to chose
an algebraic section $\eta$ of $E\times_XV\to V$ which induces
a holomorphic section $\eta'$
of $E'=E\times_XV\to V'\to V'$.

Then there exists a holomorphic map $\zeta:V'\to G$ such that
\[
\sigma'(x)=\eta'(x)\cdot \zeta(x) \forall x\in V'
\]
where $\cdot$ denotes the $G$-principal right action on
$E'=E\times_X V'$.

There is a unbounded continuous function $\rho:V'\to\R$ given by
$\rho(x)=|\tau(x)|$.

We claim that there exists sequences of discrete subsets $D_n\subset
V'$
with the following properties:
\begin{itemize}
\item
For each $n\in\N$ the projection map $\tau|_{D_n}:D_n\to\tau(D_n)$
is injective.
\item
All the sets $\tau(D_n)$ are disjoint and their union is discrete in
$\C$.
\item
For each $D_n$ the image in $V$ is Zariski-dense under the
natural map $V'=V\times_B\C\to V$.
\end{itemize}

We choose the subsets $D_n$ recursively.
Suppose $D_1,\ldots,D_{n-1}$ already chosen.
Let $C=V'\setminus\cup_{k<n}\tau^{-1}(\tau(D_k))$ and $X=V$ and
use proposition~\ref{disc-dense}. Since $\tau:V'\to\C$ has finite
fibers, we can decompose the subset $D$ thus obtained as a
finite union of $D=\Delta_1\cup\ldots\Delta_r$ such that
no fiber $\tau^{-1}(x)$ ($x\in\C$) intersects any $\Delta_i$
in more than one point. The union of the Zariski closures of the
images of the $\Delta_i$ in $V$ equals tha Zariski closure of
the image of $D$ which is $V$. Since $V$ is irreducible,
we deduce that we can choose one $\Delta_i$ for which the image
in $V$ is Zariski dense. We fix this index $i$ and choose $D_n$
as
$D_n=\{x\in \Delta_i:\rho(x)=|\tau(x)|>n\}$.

Chosen in this way, the sequence of subsets $D_n$ has the desired
pro\-per\-ties.

Next we fix a countable dense subset $\Sigma\subset G$
and a bijection $\alpha:\N\to\Sigma$.
There is surjective holomorphic map $\xi$ from some
vector space $\C^N$ onto $G$ (theorem~\ref{cn-Lie-surj}).
We choose elements $v_n\in\C^N$ such that
$\xi(v_n)=\alpha(n)$ for all $n\in\N$.

Define $\Delta=\cup_n\tau(D_n)\subset\C$.
By construction $\Delta$ is discrete in $\C$.
We can thus chose a holomorphic function
$\Phi:\C\to\C^N$ with $\Phi(\tau(x))=v_n$ for all $n\in\N$ and all
$x\in D_n$.
Define $\beta=\xi\circ\Phi:\C\to G$. Then
$\beta(\tau(x))=\alpha(n)$ for all $x\in D_n$.

Again using the surjective map $\xi:\C^N\to G$, we may chose a holomorphic function
$\gamma:\C\to G$ with $\gamma(\tau(x))=\zeta(x)$ for all $n\in\N$ and all
$x\in D_n$. 

Next we define a new section $\tilde\sigma$ in the pull-back bundle
$E\times_X\C\to\C$  as follows:
\[
\tilde\sigma(z)=\sigma(z)\cdot(\gamma(z)^{-1})\cdot
\beta(z) \quad (z\in\C)
\]
where $\cdot$ denote the right $G$-action on the principal bundle.

Then
\begin{align}
\tilde\sigma(\tau(x))&=\sigma(\tau(x))\cdot(\zeta(x)^{-1})\cdot
\beta(\tau(x))\\
&=\eta'(x)\cdot
\beta(\tau(x))
\end{align}
for all $x\in V'$ and therefore
\[
\tilde\sigma(\tau(x))=\eta'(x)\cdot\alpha(n)
\]
for all $x\in D_n$ (and all $n\in\N$).

Let $Z=\eta(V)\subset E_XV$. This is an algebraic subvariety, since
$\eta$ is algebraic. Let $\epsilon$ denote the natural map
from $E'=E_XV'$ to $E\times_XV$ and let $\chi$ denote the natural map
from $V'$ to $V$.
Then
\[
\epsilon(\tilde(\sigma((\tau(D_n)))
=\epsilon(\eta(D_n))\cdot\alpha(n)
=\eta(\chi(D_n))\cdot\alpha(n)
\] 
is Zariski
dense in $Z\cdot\alpha(n)$. As a consequence,
the (algebraic)
Zariski closure of $\cup_n\epsilon(\tilde\sigma(\tau(D_n)))$ contains all
$Z\cdot\alpha(n)$
($n\in\N$). This implies that 
$\epsilon(\tilde\sigma(\tau(V')))$ is Zariski
dense in $E'$. Therefore $\epsilon(\tilde\sigma(\C))$ induces a holomorphic map
from $\C$ to $E$ with Zariski dense image.
Hence the assertion.
\end{proof}

\begin{proposition}
Let $C$ be a Stein complex manifold, let $X$ be a complex algebraic variety
on which a connected complex Lie group $G$ acts and let $f:C\to X$
be a holomorphic map.

Then there exists a holomorphic map $\alpha:C\to G$ such that the
map $F:C\to X$ defined by $F(x)=f(x)\cdot\alpha(x)$
fulfills the following property:

The (algebraic) Zariski closure of $F(C)$ is $G$-invariant and 
contains the Zariski closure of $f(C)$.
\end{proposition}
\begin{proof}
For each holomorphic map $\alpha:C\to G$ let $F_\alpha:C\to X$
denote the map defined by $F_\alpha(x)=f(x)\alpha(x)$.
Among all these maps $F_\alpha$ we choose one for which the
dimension of the Zariski closure of the image $F_\alpha(C)$ in $X$
is maximal.
We claim that this $F_\alpha$ has the desired property.

Let $D\subset C$ be a discrete subset for which the Zariski closures
of $f(C)$ and $f(D)$ coincide.
Now assume that the Zariski closure $A$ of $F_\alpha(C)$ is not
$G$-invariant. Then there exists a point $p\in C\setminus D$ such that
$A$ does not contain the $G$-orbit through $F_\alpha(p)$.
Fix such a point $P$ and an element $g\in G$ for which $p\cdot g\not\in A$.
Next we choose a holomorphic map $\beta:C\to G$ such that
$\beta(x)=\alpha(x)$ for every $x\in D$ and such that
$\beta(p)=\alpha(p)\cdot g$. Then the Zariski closure of $F_\beta(C)$
contains $F_\alpha(D)=F_\beta(D)$. By the choice of $D$ it follows
that
the Zariski closure of $F_\beta(C)$ contains that of
$F_\alpha(C)$.
By the maximality condition and the fact that $C$ is an irreducible
space it follows that the Zariski closure of $F_\alpha(C)$ and
$F_\beta(C)$ coincide.
 This is a contradiction, because $p\cdot g$ is contained
in $F_\beta(C)$, but not in the Zariski closure of
$F_\alpha(C)$.
Hence $\overline{F_\alpha(C)}$ needs to be $G$-invariant.
\end{proof}

\section{Integral sets}
We use the notation of Lang and Vojta.

Let $X$ be a projective variety defined over a number field $K$
with algebraic closure $\bar K$,
$M$ the set of all places of $K$,
$S$ a set of places containing all the archimedean ones, and
$D$ a Cartier divisor on $X$ defined over $K$.
In this situation Lang defined (\cite{L}) 
the notion of a ``Weil function''
$\lambda_D:M\times X(\bar K)\to\R$.

A collection of real numbers $C_v$ (with $v\in M$) is called
an ``$M$-constant'' if $C_v=0$ for all but finitely many $v\in M$.

If $D$ is very ample and $i:X\setminus|D|\hookrightarrow\A^N$ a closed
embedding given by the sections of $D$, 
then a possible choice for a Weil function is
\[
\lambda_D(v,x)=\log^+\max_{k=1\ldots N}|(i(x))_k|_v
\]
where $\log^+(w)$ is defined as $\max\{0,w\}$.

A set $A\subset X(K)$ is called ``$(S,D)$-integral'' if there
is an $M$-constant $C_v$ such that $\lambda_D(p,v)\le C_v$ for
all $p\in A$ and $v\not\in S$.

\subsection{Non-compact varieties}
Let $X$ be a non-complete variety. By a result of Nagata \cite{Na}
$X$ can be embedded into a complete variety $\bar X$. Due to the
desingularization theorems in characteristic zero we may assume
that $\bar X$ is smooth and $\bar X\setminus X$ is a s.n.c.~divisor
provided $X$ is smooth.

\begin{definition}
Let $X$ be a (non-complete) smooth variety. 

A function $\mu:X(\bar k)\times M\to\R$
is called ``Weil function for $X$'' if there exists
a completion $X\hookrightarrow \bar X$
by a divisor $D=\bar X\setminus X$ and constants $C>1$, $C'>0$
such that
\[
C\lambda_D(x,v)+C' \ge \mu(x,v) \ge \frac{1}{C}\lambda_D(x,v)-C'
\]
for all $x\in X(\bar k)$.
\end{definition}

\begin{definition}
A subset $A\subset X(k)$ is called an ``$S$-integral set''
iff there is a Weil function $\lambda$ for $X$
such that
$\lambda(x,v)\le 0$ for all $x\in A$, $v\in M_k\setminus S$.

$A$ is called an ``integral set'' if there exists a finite subset
$S\subset M_k$ which includes all archimedean places for which
$A$ is an $S$-integral set.
\end{definition}

Equivalently: A subset $A\subset X(k)$ is an integral set
if there exists a compactification $X\hookrightarrow\bar X$
such that $D=\bar X\setminus X$ is a Cartier divisor and
$A$ is an $(S,D)$-subset of $\bar X$.

\begin{lemma}
Let $X$ be a variety. If $\lambda$ and $\tilde\lambda$ are
two Weil functions for $X$, then
there is a real number $C\ge 1$ and $M$-constants $a_v$ such that
$\lambda_v\le a_v+C\tilde\lambda_v$.
\end{lemma}
\begin{proof}
If $X\hookrightarrow\bar X$ and $X\hookrightarrow\tilde X$
are two completions by divisors, we may consider the diagonal
embedding $X\hookrightarrow\bar X\times\tilde X$. Let $Y$
denote the closure of $X$ in $\bar X\times\tilde X$.
Let $D=\bar X\setminus X$ and $\tilde D=\tilde X\setminus X$.
Then $p_1^*D$ and $p_2^*\tilde D$ have the same support in 
$Y$. Hence the statement.
\end{proof}

\subsection{Examples}
\begin{enumerate}

\item If $V$ is projective, any set of $K$-rational points is integral.

\item If $V$ is affine, a subset $A\subset V(K)$ is $S$-integral iff
there exists a proper embedding $i\hookrightarrow\A^N$ with
$i(A)\subset\A^N(\O_S)$.

\item If $V=\A^n\setminus\{(0,\ldots,0)\}$
an $S$-integral set is given by taking all $(x_1,\ldots,x_n)\in\O_S^n$
such that $x_1,\ldots,x_n$ generate the unit ideal of $\O_S$.

\item If $V=\P^n\setminus\{[1:0:\ldots:0]\}$
an $S$-integral set is given by taking all $[x_0:\ldots:x_n]$
such that $x_0$ is contained in $x_1\O_S+\ldots+x_n\O_S$.

\item If $V=G_m$ (the multiplicative group), the subgroup of units $\O_S^*$
is an $S$-integral set.

\item If $A$ is an $S$-integral set in a variety $V$ which admits
a morphism $f$ to a variety $W$, then $f(A)$ is a $S$-integral set.

\item
Let $\pi:V\to W$ be an \'etale covering, defined over some number field $K$
and $A$ a $S$-integral subset of $W$. Then there is a finite
field extension $F/K$ such that  $\pi^{-1}(A)\subset V(F)$. Moreover
$\pi^{-1}(A)$ is a $T$-integral set where $T$ denotes the set of all 
non-archimedean places of $F$ lying over $S$.

\item
Every finite subset is an integral set.

\item
If $A$ is a $S$-integral subset of a variety $W$ and $f:V\to W$
is a proper morphism, all defined over a number field $K$,
then $f^{-1}(A)\cap V(K)$ is a $S$-integral set.

\end{enumerate}

\section{Fibrations}
In real geometry, if $G$ is a Lie group and $H\subset I\subset G$
are closed Lie subgroups then the fibration $G/H\to G/I$ is 
``locally trivial'' in the ordinary topology.
In algebraic geometry, if $G$ is an algebraic group
and $H\subset I\subset G$ are algebraic subgroups, then
$G/H\to G/I$ is not ne\-ces\-sa\-rily locally trivial with respect to
the Zariski topology, but always locally trivial with respect
to \'etale topology as already observed in \cite{S}.
Following \cite{S}, an algebraic group is called ``special''
if such a fibration is necessarily locally trivial for the
Zariski topology. 
A non-trivial finite group can not
be special: For every such group $G$ there is a 
unramified Galois covering $X'\to X$ of complex algebraic curves
with $G$ as Galois group. This is a $G$-principal bundle which
is evidently not locally trivial in the Zariski topology.

On the other hand, all solvable connected algebraic
groups as well as certain semisimple groups like e.g.~$SL_n$
are special. However, there are simple connected algebraic
groups which are not special, e.g. $SO_n$ (\cite{S}).

From now on a $G$-principal bundle $\pi:E\to B$ is a principal bundle
which is locally trivial with respect to the \'etale topology.
In other words, there is a free right action of $G$ on $B$ such that
the $G$-orbits are the fibers of $\pi$ and for every point $p\in B$
there is a Zariski open neighbourhood $U$ and an \'etale covering
$\tau:V\to U$ such that there is a $G$-equivariant isomorphism
between $V\times_B E$ and $V\times G$.
If $p\in B(k)$, then $\pi^{-1}(p)$ is a $G$-principal
homogeneous space.
If there is a $k$-rational point $q\in V(k)$ with $\tau(q)=p$,
this principal homogeneous space is isomorphic to $G$, i.e., it
contains $k$-rational points. However, in general for an \'etale
morphism $\tau$ one can not find a $k$-rational point in $V$
over each $k$-rational point in $U$. For an easy example of this
phenomenon
consider $\tau:V\to U$ given by $V=U=\A\setminus\{0\}$ and
$\tau(z)=z^2$.

\section{Restricted topological product}
Let us recall the theory of restricted topological products.

If $(X_\lambda)$ is a family of locally compact topological
spaces, each endowed with a compact open subspace $Y_\lambda$,
then one can define a ``restricted topological product''.
The idea is to modify the direct product of all $X_\lambda$
in such a way that the resulting space is locally compact.
(In general an infinite product of locally compact spaces
is not locally compact.) As a set, the ``restricted topological
product'' contains all
\[
x=(x_\lambda)\in\Pi_\lambda X_\lambda
\]
such that $x_\lambda\not\in Y_\lambda$ for only finitely
many $\lambda$.
Thus, $X=\cup_SX(S)$ where $S$ runs through all finite subsets of
the index set and $X(S)=\left(\Pi_{\lambda\in S}X_\lambda\right)
\times  \left(\Pi_{\lambda\not\in S}Y_\lambda\right)$.
Now a locally compact topology is introduced on $X$ as follows:
A subset $U\subset X$ is open if and only if $U\cap X(S)$
is open in $X(S)$ (with respect to the product topology on $X(S)$)
for all $S$. (Warning: This is not the topology obtained by embedded
$X$ into the direct product of all $X_\lambda$.)

\begin{lemma}\label{lem-cont-sum}
Let $X$ be as above and let $\rho_\lambda:X_\lambda\to\R$
be a collection of continuous functions such that $\rho_\lambda$
vanishes in $Y_\lambda$ for almost all $\lambda$.

Then
\[
\rho(x)=\sum_\lambda \rho_\lambda(x_\lambda)
\]
defines a continuous function on $X$.
\end{lemma}
\begin{proof}
First we note that for each fixed $x\in X$ the sum on the right is
actually finite, thus $\rho:X\to\R$ is well-defined.
Let $T$ be the set of all
$\lambda$ for which $\rho_\lambda$ does not vanish on $Y_\lambda$.
Then 
\[
\rho(x)=\sum_{\lambda\in S\cup T}\rho_\lambda(x_\lambda)
\]
for every finite index set $S$ and every $x\in X(S)$.
This being a finite sum implies that the restriction of $\rho$
to each $X(S)$ is continuous.
Therefore $\rho^{-1}(W)\cap X(S)$ is open for every $S$ and every
open subset $W\subset\R$. Hence $\rho$ is continuous.
\end{proof}

\section{Adelic groups}
Let $K$ be a number field, $M$ the set of absolute values,
 $M_\infty$ the set of archimedean
absolute values and $M_f$ the set of non-archimedean ones.
As usual, for each $v\in M$ we denote by $K_v$ the completion of $K$
with respect to the absolute value  $v$.

Then the ring of Adeles $\Ad$
is defined as the restricted topological
product of all $K_v$ where the role of
the compact open subsets $Y_\lambda$ is played
by
the ring of $v$-integers $\O_v=\{x\in K_v:|x|_v\le 1\}$ if
$v$ is non-archimedean and the empty set if $v$ is archimedean.

The ring $\Af$ of ``finite Adeles'' is the restricted topological
product for only the non-archimedean valuations.

It is easy to verify that$\A$ and $\Af$ are topological rings in a 
natural way.

\subsection{Adelic groups}
Let $G$ be a linear $K$-group. Then $G$ is defined as an algebraic
subgroup of $GL_n$ as the zero-set of certain polynomials with
coefficients in $K$.
For each $v\in M$ we define $G(K_v)$ and for each $v\in M_f$
a compact open subgroup $G(O_v)=G(K_v)\cap GL_n(O_v)$ where
$O_v$ is the ring of integers of $K_v$, i.e. 
$O_v=\{x\in K_v:||x||_v\le 1\}$.
As usual, $GL_n(O_v)$ denotes the subgroup of all matrices
$A$ in $GL_n(K_v)$ such that all matrix coefficients {\em and}
$1/\det(A)$ are in $O_v$.

The adelic group is the restricted topological product of all
these groups and will by denoted by $G(\A)$ (resp.~$G(\Af)$ if we
consider the restricted topological product with respect to only
the non-archimedean absolute values of $K$).

\subsection{A finiteness theorem of Borel}
A finiteness theorem of Borel has the following consequence which we
will
need later on.
\begin{proposition}\label{borel-finite}
Let $G$ be a linear algebraic group defined over some number
field $K$, $G(\Af)$ the finite adelic group and $G(K)$ the group of
$K$-rational points, embedded diagonally into $G(\Af)$.

Then there exists a compact subset $C\subset G(\Af)$
such that $C\cdot G(K)=G(\Af)$.
\end{proposition}
\begin{proof}
The locally compact group $G(\Af)$ contains
$H=\Pi_v G(\O_v)$ as a compact open subgroup.
By a theorem of Borel (\cite{B1}) the double quotient
$H\backslash G(\Af)/G(K)$ is finite.
Hence there is a finite subset $E\subset G(\Af)$ such that
$G(\Af)=H\cdot E\cdot G(K)$.
Evidently $C=H\cdot E$ is compact.
\end{proof}

\section{Weil functions for fiber bundles}

Here we relate Weil functions on fiber bundles to Weil functions on base
and fiber of such a bundle.

\begin{proposition}\label{bundle-weil-fct}
Let $\pi:E\to B$ be a Zariski locally trivial fiber bundle.
Let $(U_i)_{i\in I}$ be a family of
 Zariski open subvarieties of $B$ which cover
all of $B$ and such that $W_i=\pi^{-1}(U_i)\to U_i$ is trivial
for every $i\in I$,
i.e., there are morphisms $\zeta_i: W_i\to F$ such that
there is an isomorphism $W_i\to U_i\times F$ given by
$x\mapsto (\pi(x),\zeta_i(x))$. Assume all defined over
some number field $K$.

Assume the index set $I$ to be totally ordered.

Let $\lambda_B$, $\lambda_F$  be a Weil functions for $B$ 
resp.~$F$ and define $\mu$
as follows:
\[
\mu(x,v)=\lambda_{F,v}\circ\zeta_i(x)
\]
if for all $j\in I$ 
either $\lambda_{U_j,v}(\pi(x))>\lambda_{U_i,v}(\pi(x))$
or $j\ge i$ and $\lambda_{U_j,v}(\pi(x))=\lambda_{U_i,v}(\pi(x))$.

Then there is a global Weil function $\lambda_E$ for $E$ such that
$\lambda_E\le\mu+\lambda_B\circ\pi$.
\end{proposition}
\begin{proof}
For each index $i$ a Weil function for $W_i$ is given by
\[
\lambda_{W_i} = \lambda_{U_i}\circ\pi + \lambda_F\circ \zeta_i.
\]
Since $E$ is covered by the open subvarieties $W_i$, a Weil function
for $E$ can be defined by
\[
\lambda_E=\inf_i\lambda_{W_i} = 
\inf_i(\lambda_{U_i}\circ\pi + \lambda_F\circ \zeta_i).
\]
It follows that
\[
\lambda_{E,v}(x)\le(\inf_i(\lambda_{U_i,v}\circ\pi(x))+  
\lambda_{F,v}\circ \zeta_j(x).
\]
for all $v$, $j$ and $x\in W_j$.
Therefore $\lambda_E\le\mu+\lambda_B\circ\pi$.
\end{proof}

\section{Weil functions and Adeles}
\begin{proposition}
Let $G$ be an algebraic group defined over a number field $K$
and 
let $E\to B$ be a Zariski locally trivial $G$-principal bundle,
likewise defined over $K$.
Let $\lambda_B$ be a Weil function for $B$.

Then there exists a Weil function $\lambda_E$ for $E$
such that for every
$x\in B(K)$ there exists a point $y\in E(K)$
with
$\lambda_B(x)\le\lambda_E(x)$, i.e.,
\[
\lambda_{E,v}(y)\le \lambda_{B,v}(x)\ \ \forall v\in M.
\]
\end{proposition}

\begin{proof}
Consider a closed embedding $i:G\to GL_n$. The matrix coefficients
of this embedding together with $\det(\ )^{-1}$ yield an embedding $\alpha$
of $G$ as a closed subvariety into $\A^{n^2+1}$.
Now we can choose a global Weil function $\lambda_G$ for $G$ via
\[
\lambda_G(v,x)=
\lambda_{G,v}(x)=\log^+\max\{|\alpha_k(x)|_v:1\le k\le n^2+1\}
\]
where $\log^+(w)=\max\{0,w\}$.
Each $\lambda_{G,v}$ extends to a function defined on $G(K_v)$
where $K_v$ denotes the completion of $K$ with respect to $v$.
By lemma~\ref{lem-cont-sum} we obtain a continuous function
$\eta$ on the adelic group $G(\Af)$ via
\[
\eta(x)=\sum_v \lambda_{G,v}(x_v).
\]
Note that
\[
\eta(x)\ge\max_v \lambda_{G,v}(x_v)
\]
because $\lambda_{G,v}(x_v)\ge 0$.

By proposition~\ref{borel-finite}
there is a compact subset $C$ of $G(\Af)$ such that
$C\cdot G(K)=G(\Af)$. Let $c'=\max\{\eta(x):x\in C\}$.
Let $d=\deg(K/\Q)$ and define $S$ as the set of all $v\in M_K$
such that $\frac{1}{d}\log p\le c'$ where $p$ is the rational prime
with $p|v$. Note that $S$ is a finite set. Hence we obtain
a $M$-constant $c_v$ by defining
\[
c_v= \begin{cases}
c' & \text{ if  }v\in S \\
0 & \text{ else }.
\end{cases}
\]
By the construction of $\lambda_G$ we know:
If there exists an element $x\in G(K)$ and a place $v\in M_K$ with 
\[
0 < \lambda_{G,v}(x)\le c',
\]
then there exists a number $z\in K$ with
\[
0 < \log|z|_v\le c'
\]
which in turn implies $v\in S$.
It follows that for $v\not\in S$ and $x\in G(K)$ the conditions
``$\lambda_{G,v}(x)\le c'$'' and ``$\lambda_{G,v}(x)\le 0$''
are equivalent.
Therefore $\eta(x)\le c'$ for $x\in G(K)$ implies
\[
\lambda_{G,v}(x)\le c_v\  \forall v\in M.
\]
Since $G(K)$ is dense in $G(K_v)$ for every $v$, the same
conclusion holds for every $x\in G(\Af)$.

Let $\mu$ be defined as in proposition~\ref{bundle-weil-fct}.
Proposition~\ref{bundle-weil-fct} states that there is a Weil
function $\hat\lambda_E$ for $E$ such that 
$\hat\lambda_E\le\mu+\lambda_B\circ\pi$.
We define $\lambda_{E,v}=\hat\lambda_{E,v}-c_v$. Recall that
the sum of a
Weil function and an $M$-constant ist still a Weil function.
Thus we found a Weil function $\lambda_E$ for $E$ with the
property 
\[
\lambda_{E,v}+c_v\le \mu_v+\lambda_{B,v}\circ\pi\ \ \forall v.
\]

Now we fix a point $x\in B(K)$ and choose a point $p\in E(K)$ with
$\pi(p)=x$. Let $(U_i)$ be a family of Zariski open subsets of $B$
as in  proposition~\ref{bundle-weil-fct} and let
$U_0,\ldots U_r$ denote those of the open sets 
$U_i$ which contain $x$. Let $\tau_i:U_i\to G$ be defined as in
proposition~~\ref{bundle-weil-fct}.
Then there are elements $g_0,\ldots,g_r\in G$ with
\[
\tau_i(q)= \tau_0(q)\cdot g_i
\]
for all $0\le i\le r$ and $q\in E(K)$ with $\pi(q)=x$.

There is a map $\xi:M_K\to\{0,\ldots,r\}$
such that
\[
\mu(q,v)=\lambda_{G,v}(\tau_{\xi(v)}(q)).\ \forall q\in\pi^{-1}(x)\cap E(K).
\]
We define an element $a=(a_v)_v$ in the finite 
adelic group $G(\Af)$ as follows: 
$a_v=g_{\xi(v)}\cdot \tau_0(p)=\tau_{\xi(v)}(p)$.
Then
\[
\mu_v(p\cdot g)= \lambda_{G,v}(\tau_{\xi(v)}( p\cdot g))
=\lambda_{G,v}(a_v\cdot g)\ \forall g\in G(K).
\]

Next we choose $g\in G(K)$ such that $a\cdot g\in C$
and define $y=p\cdot g$.
Then $\eta(a\cdot g)\le c'$ which in turn implies
\[
\mu_v(y)=\mu_v(p\cdot g)=\lambda_{G,v}(a_v\cdot g)\le c_v\ \forall v\in M_K
\]
It follows that
\[
\lambda_{E,v}(y)\le \lambda_{B,v}(\pi(y)) + \mu_v(y) - c_v 
\le \lambda_{B,v}(\pi(y))=\lambda_{B,v}(x)
\]
as desired.
\end{proof}

\begin{corollary}\label{cor4}
Let $G$ be an algebraic group
and let $\pi:E\to B$ be a Zariski locally trivial $G$-principal bundle,
 all defined over some number field
$K$. Let $S$ be a finite set of places of $K$.

Let $A$ be a $S$-integral subset of $B(K)$ and assume that 
$E(K)\ne\emptyset$.

Then there exists a $S$-integral subset $A'\subset E(K)$ with
$\pi(A')=A$.
\end{corollary}
\begin{proof}
By assumption there is a $G$-principal homogeneous space $F$
such that
every point $p\in B(K)$ admits a Zariski open neighbourhood $U$
for which $\pi^{-1}(U)(K)\simeq U(K)\times F(K)$.
Therefore the condition $E(K)\ne\emptyset$ implies
$F(K)\ne\emptyset$ which in turn implies that $\pi(E(K))=B(K)$.

Hence the assertion.
\end{proof}

\begin{proposition}\label{prop7}
Let $G$ be an algebraic group defined over a number field $K$.

Then there exists a finite field extension $L/K$ such that
$G(L)$ is Zariski dense in $G$.
\end{proposition}
\begin{proof}
As usual, let $\bar K$ denote an algebraic closure of $K$.
Let us consider all algebraic subgroups which arise as the
Zariski connected component of a Zariski closure of
a finitely generated subgroup of $G(\bar K)$. Let $H$ be maximal among
these algebraic subgroups. One verifies easily that $H$ is normal
and that every element in $(G/H)(\bar K)$ is of finite order.
This implies that $G/H$ contains no algebraic subgroup isomorphic
to the additive group $G_a$ or the multiplicative group $G_m$.
As a consequence, the connected component $A$ of $G/H$ must be
an abelian variety. Unless $A$ is finite, it contains, like every
projective variety, $\bar K$-rational points of non-zero height.
However, as an abelian variety, $A$ admits a canonical height,
the Neron-Tate height, which is zero for every torsion element.
Therefore $A(\bar K)$ must contain an element of infinite order
unless $\dim A=0$. Thus $G/H$ must be finite. Now we can chose
$L$ such that it contains the fields of definition for one
point in each connected component of $G$ and for each element
in a finite set of generators for some subgroup whose Zariski
closure contains $H$. Then $G(L)$ is Zariski dense in $G$.
\end{proof}

\begin{remark}
If $G$ is linear and connected, then it is not necessary to
enlarge the field. Rosenlicht proved in \cite{R}
that in this case $G(K)$
is already Zariski dense.
\end{remark}

\begin{proposition}\label{G-dense-integral}
Let $G$ be an algebraic group defined over some number field $K$.
Then there exists a finite field extension $L/K$ such that
$G(L)$ contains a Zariski dense $S$-integral subset where
$S$ denotes the set of all non-archimedean places of $L$.
\end{proposition}

\begin{proof}
Let $Z$ denote the center of $G$.
By the theorem of Chevalley there is a unique maximal linear 
connected subgroup $H$ of $Z$ and furthermore the quotient $Z/H$
is an abelian variety. As the center of $G$, $Z$ is invariant under
the action of the Galois group of $K$ and therefore $Z$ is defined
over $K$. Similarily uniqueness of the maximal linear subgroup $H$
implies that $H$ is defined over $K$.

The quotient group $G/Z$ is linear, because the center $Z$ is the
kernel of the adjoint representation of $G$. 
Over the algebraically closed field $\bar K$ every connected
linear algebraic group is generated by its one-dimensional 
connected algebraic
subgroups. Moreover, every linear connected one-dimensional $\bar K$-group
is isomorphic to either the additive group $G_a$ or the
multiplicative group $G_m$. Hence there is a finite field
extension $K'/K$ such that $G/Z$ is generated by one-dimensional
subgroups defined over $K'$ each isomorphic to $G_a$ or $G_m$
over $K'$. Therefore there is a dominant $K'$-morphism from
some $G_a^n\times G_m^l$ to $G/Z$. This implies that there
is a Zariski dense integral subset $E$ in $G/Z$ 
(see 3.2.(6)).
Similarily, we may assume that the connected
linear group $H$ contains
a Zariski dense integral subset $E_H$.
By a further field extension $K''/K'$ we may assume that
both $G$ and $Z/H$ admit Zariski dense sets of
$K''$-rational points
(proposition~\ref{prop7}).
Now the projection $\tau:G/H\to G/Z$ is a proper morphism
(because $Z/H$ is an abelian variety); hence
$E'=\tau^{-1}(E)\cap(G/H)(K'')$ is an integral subset of $G/H$.
Moreover $E'$ is Zariski dense in $G/H$, because $(Z/H)(K'')$
is Zariski dense in $Z/H$ and $E$ is Zariski dense in $G/H$.
By definition, the algebraic group $H$ is connected, commutative
and linear and therefore ``special'' in the sense of Serre
(\cite{S}, proposition~14). For this reason, the quotient map $G\to G/H$
is a Zariski locally trivial fiber bundle and we may deduce
from corollary~\ref{cor4} that there is an integral subset $E''$ of $G(K'')$
which surjects onto $E'$. Finally we take the image of $E''\times
E_H$ under the morphism $G\times H\to G$ given by group
multiplication in order to get a Zariski dense integral
subset of $G$.
\end{proof}

\section{Ramified coverings over semi-abelian varieties}
Here we prove theorem~\ref{thm-ram}.
\begin{proof}
First we observe that $\pi:X\to A$ fibers through the quasi-Albanese variety of $X$.
Since $X$ is defined over $K$, so is its quasi-Albanese.

A semi-abelian variety admits only countably
many semi-abelian subvarieties. As a consequence, if a semi-abelian
variety is defined over $K$, then every semi-abelian subvariety is
defined over some finite field extension of $K$. It follows
that quotient of the given semi-abelian variety by this semi-abelian
subvariety is likewise defined over $K$.

Applied to the quasi-Albanese $Q$ of $X$ and in consideration of the
morphism from $Q$ to $A$ induced by $\pi:X\to A$ these arguments
show that
there is a finite extension $K_0$ of $K$ such that $A$ is defined
over $K_0$.

By \cite{NWY} there is a semi-abelian variety $B$, a morphism
$\phi:B\to X$ and a holomorphic map $F:\C\to B$ such that
$f=\phi\circ F$ and such that $\phi(B)$ equals the Zariski
closure of $f(\C)$ in $X$.
Furthermore we may assume that the induced map 
$\pi\circ\phi:B\to A$
is finite.
Then the above arguments on semi-abelian subvarieties imply
that  $B$ is defined over some finite field extension
$K_1$ of $K_0$. 

Now $\pi\circ\phi:B\to A$ is a morphism between two semi-abelian
varieties which are both defined over $K_1$.
It follows that there is an element $a\in A(\C)$ and a morphism
of $K_1$-group varieties $\zeta:B\to A$ such that $\pi\circ\phi$
and $\zeta$ differ by the map $\tau_a$ given
by translation by $a$.

We regard the space of all morphisms $\psi$ from $B$ to $X$ with the
property that there is an element $p\in A$ such that
$\pi\circ\psi=\tau_p\circ\zeta$.

This is an algebraic variety which contains a $\C$-rational point.
Thus it also contains a $\bar\Q$-rational point and consequently
it contains a rational point for some finite field extension $K_2$
of $K_1$.

This yields a morphism $\psi:B\to X$ defined over $K_2$
whose image $\psi(B)$ has the same dimension as $\phi(B)$.

Next we chose a finite field extension $K_3/K_2$ such that
$B(K_3)$
admits a Zariski dense integral subset $E$.
(This is possible due to proposition~\ref{G-dense-integral}.)
Then $R=\phi(E)$ is a integral subset of
$X(K_3)$ whose Zariski closure has the same dimension
as the Zariski closure  of $f(\C)$.
\end{proof}

\end{document}